\documentclass[10pt]{amsart}

\usepackage{amsmath}
\usepackage{amssymb}
\usepackage{amsthm}

\usepackage[dvips]{graphics}
\usepackage{graphicx}
\usepackage{epsfig}
\usepackage{hyperref}
\usepackage{ifthen}

\newcommand{\Rr}{{\mathbb{R}}}

\newcommand{\Zz}{{\mathbb{Z}}}

\newcommand{\Tt}{{\mathbb{T}}}

 \def \Bbb{\mathbb}
 \def \cal{\mathcal}

 \def\-lim{{\text{\rm -lim}}}

 \def\nulo{\diameter}
 \def\0{\nulo}

 \theoremstyle{plain}
 \newtheorem{Thm}{\bf Theorem}[section]
 
 \newtheorem{lemma}[Thm]{\bf Lemma}

 \newtheorem{theorem}[Thm]{\bf Theorem}
 \newtheorem{proposition}[Thm]{\bf Proposition}

 \theoremstyle{definition}
 \newtheorem{Definition}[Thm]{\bf Definition}

\newcommand{\Hh}{{\overline{H}}}

\newcommand{\Ll}{{\overline{L}}}
\newcommand{\Cg}{{\mathcal{C}}}

\newcommand{\Mm}{{\mathcal{M}}}

\newcommand{\Rm}{{\noindent \sc Remark. \ }}

\theoremstyle{definition}

\begin{document}
\title{Minimax probabilities for Aubry-Mather Problems }
\author{Diogo A. Gomes, Nara Jung and Artur O. Lopes}

\begin{abstract}
In this paper
we study minimax Aubry-Mather measures and its main properties.
We consider first the discrete time problem and then the continuous time case.
In the discrete time problem we establish existence,
study some of the main properties using duality theory
 and present some examples.
In the continuous time case, we establish both existence and non-existence results.
First we give some examples that show that in continuous time stationary minimax Mather measures
are either trivial or fail to exist. A more natural definition in continuous time
are $T$-periodic minimax Mather measures. We give a complete characterization of these measures and discuss several examples.
\end{abstract}

\maketitle

\thanks{D. Gomes was
partially supported by the Center for Mathematical Analysis,
Geometry and Dynamical Systems through FCT Program POCTI/FEDER and
also by grant DENO/FCT-PT (PTDC/EEA-ACR/67020/2006). A. O. Lopes \, was \,
partially \, supported by CNPq, PRONEX -- Sistemas Din\^amicos,
Instituto do Mil\^enio, and is beneficiary of CAPES financial
support. }


\section{Introduction}

The main purpose of the present paper is to define and  analyze some basic properties of minimax Aubry-Mather measures.
In the first part of this paper we consider a minimax analog of the discrete-time Aubry-Mather theory (see \cite {Gom2} and also the last section of \cite{GLM}) and in the second part the continuous time case (see \cite{Mat} \cite{BG} \cite{CI} \cite{Fa}).
As a motivation to study minimax Mather measures, consider, in the continuous time problem, the following  one-dimensional Lagrangian:
\[
L=\frac{v^2}{2}\,-\,U(x), \,\,\,\ x \in S^1.
\]
Suppose that $U$ has a point of maximum $x_M$ and a point of minimum
$x_m$. The (minimizing) Mather measure is simply $\delta(x-x_M)\delta(v)$,
where $\delta(v)$ is the Dirac delta on $v=0$. According to the definition of
minimax Mather measure given later in the paper, the minimax measures for this
Lagrangian is $\mu=\delta(x-x_m)\delta(v)$. This measure is more natural from the point of view of the physical problem
as it is supported in the minimum of the potential energy.


We point out that some authors  consider previously minimax orbits (instead of measure like here) \cite{Mat1} \cite{LV}.

In the last few years the study of global minimizers has
been an extremely active research area and is the main focus of the
so called Aubry-Mather theory (see \cite{CI}, \cite{Fa} and \cite{BG}).
In this setting, one replaces the problem of determining orbits
that minimize the action with the problem of finding measures
$\mu(x,v)$ which minimize the action
\[
\int L(x,v) d\mu
\]
and satisfy certain holonomy constraints. These measures, which are  invariant under
the Euler-Lagrange flow,
give rise to global minimizing orbits,
and are extremely important
in understanding qualitative features of the dynamics.


In this paper we work both in the discrete and continuous time setting. We assume the following hypothesis on the Lagrangian:
$L(x,v):\Tt^n\times \Rr^n\to \Rr$, where $\Tt^n$ is the $n$-torus, identified with $\Rr^n$ when convenient, in this case $L$ is $\Zz^n$ periodic in $x$, that is
$L(x+k,v)=L(x,v)$, for all $k\in \Zz^n$. We assume further that $L$ is smooth, strictly convex in $v$:
\[
D^2_{vv}L(x, v)\geq \gamma>0,
\]
for some constant $\gamma$, and coercive (also called super-linear), that is,
\[
\lim_{|v|\to \infty}\frac{L(x, v)}{|v|}=\infty.
\]

Remember that:

\begin {theorem}
Consider $\mathbb{X}=TM$ and  $X=M$, $\mu$ a probability measure over $TM$,
$\pi_1: TM \to M$, such that $\pi_1(x,v)=x$, and
$\theta=(\pi_{1})_{*}(\mu)$. Then there exists a family of
probabilities $\pi(x,v)=\{\pi\}_{x \in M}$ over $TM_x$, uniquely
determined
$\theta$-a.e., such that,\\
1) $\pi_{x}(TM \backslash \pi_1^{-1}(x))=0$, $\theta$-a.e.;\\
2) $\int g(x,v) \mu (dx,dv) = \int_{M}\int_{\pi_1^{-1}(x)}
g(x,v)d\pi(x,dv) \theta(dx)$.
\end{theorem}
Such decomposition is called disintegration of the probability
$\mu$ (see \cite{DM} III-70 for a proof). \vspace{0.3cm}

Here, any probability $\mu(x,v)$ in the tangent bundle of the torus will be taken in a disintegrated form $\mu(x,v) = \theta (x) \pi(x,v)$. Using this point of view, we can restate the classical  Aubry-Mather problem in a way that will be suitable
for generalization.

Fix $Q\in \Rr^n$, the rotation vector.
Mather's discrete problem (see, for instance \cite{Gom2})  consists in minimizing
\[
\Ll(Q)=
\inf_{\theta}\,\,\{\,\inf_{\pi\,\in\, \Pi\,(\,\theta \, ,\, Q\,)}\int \,L(x,v)\, \pi(x,dv)\, \theta(dx)\,\},
\]
in which the infimum
is taken over all probability measures $\theta(x)$
supported in $\Tt^n$, and $\Pi(\theta, Q)$ denotes the set of
Borel measures $\pi(x,v)$ such
that, for
each fixed $x\in \Tt^n$, we have that $\pi(x,v)$ are probability measures on $v$ which satisfy
the following two constraints
\begin{equation}
\label{c1}
\int\, \phi(x+v)\, \pi(x,dv)\,\theta(dx)\,=\,\int\phi(x) \,d\theta(dx),
\end{equation}
and
\begin{equation}
\label{c2}
\int \,v \,\pi(x,dv)\,\theta(dx)\,=\,Q,
\end{equation}
with  $Q\in \Rr^n$.

The first constrain is called the holonomic constraint. The second is called the homological constraint.
We point out that the difference between  the discrete and the continuous Mather problem is  the homolonic constrain (compare (1) above with Definition 7.1).

We point out that in the classical (continuous) Mather problem the minimization of the Lagrangian on the holonomic probabilities supported on the tangent bundle is realized by a probability which is invariant with respect to the associated Euler-Lagrange flow \cite{CI} \cite{Fa} \cite{BG}.

In optimal transport theory, each element $\pi\in \Pi (\theta, Q) $ is
usually called an admissible plan.
The constraint \eqref{c2} imposes a fixed average rotation number $Q$
of the plan $\pi$ with respect
to $\theta$. The function $\Ll(Q)$ is called
the effective or averaged Lagrangian. This problem
is a discrete version of the standard Aubry-Mather problem (see \cite{CI} \cite{Fa}, for instance)
but the
notation is more convenient for our purposes, as
it is easier to generalize this problem to the minimax case, as we
will see.

In the first part of this paper we propose to study
a problem closely related to the discrete Mather problem,
the minimax problem:
\[
\hat L(Q)=
\sup_{\theta}\,\,\{\,\inf_{\pi\,\in\, \Pi\,(\,\theta, \,Q\,)}\int L(x,v)\, \pi(x,dv)\,\theta(dx)\,\},
\]
and investigate its connection with the discrete Lagrangian
dynamics.

\begin{Definition}
Given a vector $Q\in \Rr^n$
a $Q$-minimax measure  $\mu=\pi\,\theta$, $\pi\,\in\, \Pi\,(\,\theta, \,Q\,)$,
 is a probability measure on the tangent bundle, such that,
\begin{enumerate}
\item
$\int L\,d\mu\leq \int L d (\theta \tilde \pi)$, for all $\tilde \pi\,\in\, \Pi\,(\,\theta, \,Q\,)$.
\item
For any probability measure $\tilde \theta$,
\[
\int L\,d\mu \geq \inf_{\tilde \pi \in \Pi(\tilde \theta, Q)}\int L d (\tilde \theta \tilde \pi)
\]
\end{enumerate}
If $\mu$ is a $Q$-minimax measure, we define
\[
\hat L(Q)=\int L\,d\mu.
\]

\end{Definition}

In the discrete time setting, we will prove
the existence of minimax
measures $\mu=\pi \,\theta$,
and give a variational characterization
of $\hat L$ in terms of the dual problem. Later, we will consider the continuous time minimax problem.
We will give explicit examples of non-existence of minimax measures. Then, building upon the ideas in \cite{BB}
we study time-periodic minimax Mather measures, which from the point of view of the dynamics are interesting
objects related to minimax periodic orbits.



The semiclassical limit of
the Schrodinger operator and its connections with Aubry-Mather measures were investigated in
\cite{L1} and \cite{Gom1}. However, in the semiclassical limit setting,
minimax Mather measures may in fact be a more natural object, for instance, if one considers Wigner measures associated to
the ground state eigenfunction of the Schrodinger operator one obtains the minimax Aubry-Mather measure as the weak limit.

The plan of the paper is as follows: after some formal calculations in the next section, we present in section 3 
some examples of minimax measures.
In section 4 we show the existence of  the minimax measure in the general discrete time case.
Sections 5 and 6 consider duality and semi-convexity for the  minimax measure in the discrete time case. In section 7
we prove the non-existence of stationary Mather measures, for continuous time problems, whereas in section
8 we show the existence of the minimax periodic Mather measures. We introduce the concept of $T$-minimax probability, for a real $T>0$ fixed. This will be a family of probabilities $\rho(x,v,t)$ on the tangent bundle, indexed by $t\in[0,T]$. Finally, in section 9 we present some additional examples in continuous time.

As we will show, the $T$-periodic minimax Mather measures contain $T$-periodic minimax orbits. As $T\to \infty$ the minimax periodic orbits may converge to heteroclinic or homoclinic connections between minimizing orbits. Therefore, in some sense, they contain relevant information about the connection structure between minimizing orbits.

Finally, we point out a different, but analogous line of problems. Recently  in partial differential equations
minimax problems were also considered \cite{RS} \cite{AM} \cite{Be} \cite{LV1} \cite{MMW}. 

\section{Formal computations}

To get some insight on the minimax problem,
we start by performing some formal computations. First, we
introduce Lagrange multipliers $u$ and $P$ for the constraints
(\ref{c1}) and (\ref{c2}) (see \cite{Gom2} for the Aubry-Mather setting). The Lagrange multiplier $u$ for the constraint
(\ref{c1}) is a continuous function $u: \Tt^n\to \Rr$; for the constraint \eqref{c2}, we take  $P\in \Rr^n$.
We have the identity
\begin{multline*}
\hat L(Q)=
\sup_{\theta}\inf_{\pi\in \Pi(\theta,Q)}\sup_{P,u}
\int \left[L(x,v)+\right.\\
\left.u(x+v)-u(x)+P\cdot(v-Q)\right]\pi(x,dv)\theta(dx).
\end{multline*}
Applying the minimax principle to the last expression (which we will prove using duality theory
in section \ref{dualsec}),  that is,
exchanging the infimum with the last supremum,
we obtain:
\begin{multline*}
\hat L(Q)=
\sup_{\theta}\,\,\sup_{P\, , \,u}\,\,\inf_{\pi\,\in\, \Pi\,(\,\theta\,,\,Q\,)}\,
\int \left[L(x,v)+u(x+v)\right.\\ \left.-
u(x)+P\cdot(v-Q)\right]\pi(x,dv)\theta(dx).
\end{multline*}
This identity
implies that $\pi(x,v)$ is supported at points $v$ which minimize
\[
L(x,v)+u(x+v)+P\cdot v.
\]
Thus, if $u$ is differentiable,
\begin{equation}
\label{fe}
P+D_vL(x,v)+D_xu(x+v)=0,
\end{equation}
$\pi$-almost everywhere for $\theta$ almost every $x$.
Furthermore
\begin{multline*}
\hat L(Q)=
\sup_{u,P}\,\{-P\cdot Q\\+\sup_\theta
\int \inf_v\left[L(x,v)+u(x+v)-u(x)+P\cdot v \right]\theta(dx)\,\}.
\end{multline*}
This last identity yields
 that $\theta$ must be supported at the maximizers of
\[
\inf_v\left[L(x,v)+u(x+v)-u(x)+P\cdot v \right].
\]
Consequently, if we assume again $u$ to be differentiable,
\begin{equation}
\label{se}
D_xL(x, v^*)+D_xu(x+v^*)-D_xu(x)=0,
\end{equation}
$\theta$ almost everywhere, in which $v^*$ satisfies (\ref{fe}) and we assume
that it is a $C^1$ function of $x$.

We can define the Hamiltonian $H$ corresponding to $L$ by the Legendre
transform
\begin{equation}
\label{lt}
H(p,x)=\sup_v\, \{ - p\cdot v -L(x, v)\} .
\end{equation}
The Hamiltonian is smooth, periodic in $x$,
coercive,
and, as we assume that $L$ is strictly convex on $v$, we have that $H$ is strictly convex in $p$.

\medskip

\Rm
We consider in \eqref{lt}, the Optimal Control definition for the Hamiltonian.
In Classical Mechanics the
Hamiltonian is usually defined as
\[
\check{H}(p, x)=\sup_v (\, p\cdot v-L(x, v)).
\]
These two definitions differ  by the sign of  $p\cdot v$.  Therefore,
if we replace  $L(x,v)$ by the symmetrical Lagrangian, i.e., $\check{L}(x,v)=L(x,-v)$, then
$$
\check{H}(p,x)=\max_v \{p\cdot v - \check{L}(x,v)\}=\max_v \{-p\cdot v - L(x,v)\}.
$$


Observe that \eqref{se} can be written as
\[
(P+D_xu(x+v^*))-(P+D_xu(x))=D_xH(P+D_xu(x),x).
\]
Therefore, if we define
\[
x_{n+1}-x_n=v^*(x_n),
\]
and $p_n=P+D_xu(x_n)$ then $(x_n,p_n)$ satisfies the discrete form
of Hamilton's equations:
\begin{equation}
\label{disdin}
p_{n+1}-p_n=D_xH(p_{n+1},x_n)\qquad x_{n+1}-x_n=-D_pH(p_{n+1},x_n).
\end{equation}

Consider the measure in $\Tt^n\times \Rr^n$
which projects in $\Tt^n$ to $\theta$ and
is supported in the graph
\[
(x,p)=(x,P+D_xu(x)).
\]
Then this measure is invariant under (\ref{disdin}), as we reinterpret (\ref{fe}) and (\ref{se})
appropriately.

Finally, let
\[
-\hat H(P)=\sup_{u\in C(\Tt^n),x\in \Tt^n}\,\{
\inf_v\left[L(x,v)+u(x+v)-u(x)+P\cdot v\right]\,\}.
\]
Then
\[
\hat L(Q)=\sup_P\,\{ -P\cdot Q-\hat H(P)\},
\]
which shows that $\hat L$ is a convex function.

Bellow $\Hh(P)$ denotes the Legendre dual of $\Ll(Q)$.

We will record, for future reference some elementary properties of
$\hat L$ and $\hat H$:
\begin{proposition}
Suppose $c|v|^2-C\leq L(x,v)\leq C|v|+C$, for suitable constants $c, C$.
We have
\begin{enumerate}
\item $\Ll(Q)\leq \hat L(Q)$
\item $\Hh(P)\geq \hat H(P)$
\item $-C+c|Q|^2\leq \hat L(Q)\leq C+C|Q|^2$
\item $-C+c|P|^2\leq \hat H(P)\leq C+C|P|^2.$
\end{enumerate}
\end{proposition}
\begin{proof}
The first and second items are obvious from
the definition of $\hat L(Q)$ and Legendre transform, respectively.

To prove the third item, recall the well known fact
\[
\Ll(Q)\geq -C+c|Q|^2,
\]
which immediately implies $\hat L\geq -C+c|Q|^2$.
This yields
\[
\hat H(P)\leq C+C|P|^2.
\]
To prove the other
inequality, observe that it is enough to show that
$\hat H(P)\geq -c+c|P|^2$.

Given a vector   $P\in \mathbb{R}^n$, and $C\geq 0$, denote by $v=[-P]$ the vector in $\mathbb{Z}^n$
which minimizes $C|v|^2+P \cdot v$ (although this vector may not be unique, this is irrelevant for our purposes).
 For large $P$ we have $C|[-P]|^2+P\cdot [-P]\leq  -c|P|^2$, for some $c>0$.

We have
\begin{align*}
-\hat H(P)&=\sup_{u\in C(\Tt^n),x\in \Tt^n}
\inf_v\left[L(x,v)+u(x+v)-u(x)+P\cdot v\right]\\
&\leq \sup_{u\in C(\Tt^n), x\in \Tt^n}\inf_v \left[
C+C|v|^2+u(x+v)-u(x) +P\cdot v
\right]
\\
&\leq C - c|P|^2,
\end{align*}
by setting $v=[-P]$,
which then implies the remaining inequality for $\hat L$.
\end{proof}

\section{Examples}

We will show in the following sections the existence of a minimax probability measure
with rotation $Q$  for the class of strictly convex and super-linear  Lagrangians.
Before proceeding, however we
present some examples.

Consider the one-dimensional case in which the Lagrangian is
\[
L=\frac{v^2}{2}\,-\,U(x).
\]
In this case $U(x)$ is the potential energy of the corresponding problem in Classical Mechanics problem.

Suppose that $U$ has a point of maximum $x_M$ and a point of minimum
$x_m$. The (minimizing) Mather measure
with rotation number $Q=0$ is simply $\delta(x-x_M)\delta(v)$,
where $\delta(v)$ is the Dirac delta on $v=0$ (see \cite{CI} \cite{Fa}).

We claim
that the minimax Mather measure for this Lagrangian is $\mu=\delta(x-x_m)\delta(v)$, when the  rotation  constrain is $Q=0$.
Furthermore, the plan $\pi(x,v)=\delta(v)$ is clearly optimal.
Indeed, suppose by contradiction,
that any other measure $\mu (x,v)= \theta(x) \, \pi(x,v)$ is given. For a fixed $x$, the
optimal plan is, of course, $\pi(x,v)=\delta(v)$, that is, for any other plan $\tilde \pi$ we have
\[
\int L \pi(x, dv)\leq \int L \tilde \pi (x, dv).
\]
Then
$$
\int_{\Tt^n\times \Rr^n} L (x,v)\,\delta(v)\,\theta(x)\,dx \,dv=$$
$$
\int_{\Tt^n} U(x)\,  \theta(x)\, dx\leq \max U=\int_{\Tt^n\times \Rr^n} L(x,v) \, \delta(x-x_m)\,\delta(v)\,dx\,dv    .
$$
Next we consider the case of nonzero rotation number.
Let us assume that
$0<Q<1$.
Suppose that $\theta(x)= \delta(x-x_m)$. We claim that the support of $\pi(x_m,dv)$ is contained in the set of points of the form $x_m + k$ with $k$ integer. Without loss of generality suppose $x_M=0$.
Then, considering $\phi(x) =e^{2 \, \pi i \,x}$, we get  $\int e^{2 \pi i v} \pi(0,dv) =\int \pi(0,dv)=1$
The first integral is a convex combination of points in the boundary of the complex unit disk. Since
all these points are extreme points and also so is 1 it follows easily that $\pi(0,dv)$ is supported
on the integers.

Define
\[
\mu_Q=(1-Q) \delta(x-0) \delta(v)+Q  \delta(x-0) \delta(v-1).
\]
An easy computation shows that (\ref{c1}) and (\ref{c2})
are satisfied. For any fixed $x$,
the plan $\pi(x,dv)=(1-Q)\delta(v)+Q\delta(v-1)$ is optimal.
We claim that the measure $\mu_Q$ is the min-max Mather measure.
Indeed, given any other measure $\mu=\theta\, \pi$ (with rotation vector $Q$) define
\[
\tilde \mu_Q=\theta(x)\left((1-Q)\delta(v)+Q\delta(v-1)\right).
\]
It is clear that $\tilde \mu_Q$ satisfies (\ref{c1}) and (\ref{c2}) and
\[
\int L d\tilde\mu_Q\leq \int Ld\mu_Q.
\]
From this follows the claim.

We say that a probability $d\mu(x,v)$ on the tangent bundle has the graph property, when for almost every point $x$ in the projection of the support of $\mu$, we have that the $v$ on the fiber over $x$, which puts $(x,v)$ in the support of $\mu$, is unique. It follows from the convexity assumption over $L$ that Mather measures on the tangent bundle have the graph property \cite{Mat} \cite{BG} \cite{CI}.

Therefore, in the case $0<Q<1$, the graph property is not true.






\section{Existence of mini-max probabilities: discrete time case}

Let $P(\Tt^n)$ be the set of probability measures on the $n$-dimensional
torus $\Tt^n$. In addition to the previous hypothesis on $L$ we assume further that
\begin{equation}
\label{starassup}
|D_x \,L(x,v)|\leq C\ \text{and}\ |D^2_{xx} L(x,v)|\leq C.
\end{equation}
For each $\theta\in P(\Tt^n)$, define
\[
g_Q(\theta):=\,\min_{\pi\in {\Pi}[\theta,Q]}
\int_{\Omega}L(x,v)\,\theta(dx)\, \pi(x,dv),
\]
where $\Omega={\Tt^n\times \Rr^n}$, and ${\Pi}[\theta,Q]$
is the set of all plans $\pi$ which satisfy
for all continuous functions $\phi\,:\Tt^n\rightarrow \Rr$,
\[
\int_{\Omega}[\,\phi(x+v)-\phi(x)\,]\, \pi(x,dv) \,\theta(dx) = 0, \text{and}
\int_{\Omega} v\ \pi(x,dv) \,\theta(dx) =Q.
\]
If there is no rotation vector constraint $Q$, we will just write $g(\theta)$.
For a fixed $\theta$, the minimizer $\pi=\pi_\theta$ for $g(\theta)$ clearly exists (by the assumptions we made on $L(x,v)$ on the variable $v$).

Define
$$
\hat L(Q)=
\sup_{\theta}\,\,\{\,\inf_{\pi\,\in\, \Pi\,(\,\theta, \,Q\,)}\int L(x,v)\, \pi(x,dv)\,\theta(dx)\,\}=\sup_{\theta} g_Q(\theta).
$$

Consider a sequence $\theta_n$ such that $ \lim_n g(\theta_n) = \hat L(Q).$

We can assume that
\[
g(\theta_n)=\int_{\Omega}L(x,v)\,\theta_n (dx) \,\pi_n(x,dv).
\]
 Remember that for a fixed $\theta_n$ the minimizer $\pi_n$ for $g(\theta_n)$ exists. One can consider weak limit of the probabilities $\theta_n$ over $\Tt^n$ and getting in this way limits denoted generically by $\theta$.  The main point bellow is to show that $g(\theta_n) \to g(\theta)$, where $\theta_n$ is one of this subsequences.
Now, given a certain $\theta$ there exist a minimizer $\pi$ for $g(\theta)$. Then, $\mu=\theta \, \pi$ is a minimax probability.

In order to show that  $g(\theta_n) \to g(\theta)$,
we introduce a metric $d(\theta_1,\theta_2)$, for     $ \theta_1,\theta_2\in       P(\Tt^n)$,  which is a simple variation of the usual Wasserstein metric \cite{Ambro}, \cite{Ra}. We will show that $g$ is continuous with respect to this metric.

By definition, $\Pi[\theta_1,\theta_2]$ is the set of probabilities $\mu(x,v)$ on the tangent bundle such that
\begin{enumerate}
\item
$
\int_{\Tt^n\times \Rr^n} \phi(x+v)\ d\mu (x,v)=\int \phi(x)\ d\theta_2(x)
$,
\item
$
\int_{\Tt^n\times \Rr^n} \phi(x)\ d\mu(x,v) = \int \phi(x)\ d\theta_1(x)$.
\end{enumerate}

Condition 2) above means that the marginal of $\mu$ on the x coordinate is $\theta_1$, that is,
$d\, \mu(x,v) = \theta_1(dx) \, \pi(x,dv).$

We explain now condition 1) when $x$ is one dimensional.
Let $\hat \theta_2$ be the measure in $\Rr$ we obtain
if we project the probability $\mu(x,v)$, with $x \in \Tt^1$ and $v\in \Rr^1$ (the infinite cylinder),
on the coordinate  $v      \in  \Rr^1    $ (that is, in the set $0\times \Rr^1$), through lines parallel to the diagonal. Then, by considering
 the probability $\hat{\theta}_2$ (mod\, 1), that is,
on $\Tt^1$, then we obtain $\theta_2$.

We say that $\theta_1$ is the marginal of $\mu$ in the first coordinate and $\theta_2$ is the (projected via diagonal) marginal of $\mu$ in the second coordinate.

Remark: We point out that due to the homological constrain given by $\int_{\Tt^n\times \Rr^n} \phi(x+v)\ d\mu (x,v)=\int_{\Tt^n\times \Rr^n} \phi(x)\ d\mu(x,v)$, for any $\varphi$, considered  in the
discrete Aubry-Mather theory \cite{Gom2} (minimization of the Lagrangian action among probabilities $\mu(x,v)=\theta(dx)\, \Pi(x,dv)$, with $(x,v)$ in the tangent bundle of $\Tt^n$), one  can be consider this problem
(via  projection through lines parallel to the diagonal) as a kind of transshipment minimization problem for the $\mu(x,y)$, with $x,y\in \Tt^n$, which have the same marginal $\theta$ on $x$ and $y$ variables. In the notation described above we have $\mu\in \Pi(\theta,\theta)$, but $\theta$ is free to move. Due to the homological condition the minimizing $\theta$ on both problems are the same.

\begin{Definition} Consider the metric
$$
d(\theta_1, \theta_2):\ = \inf_{\mu \in \Pi[\theta_1, \theta_2]}
\int_{\Tt^n\times \Rr^n} \,\frac{|v|^2}2\,\, d\mu(x,v)=$$
$$ \inf_{\mu \in \Pi[\theta_1, \theta_2]}
\int_{\Tt^n\times \Rr^n} \,{|v|^2\over 2}\, \,\theta_1(dx) \,\, \pi(x,dv).
$$
\end{Definition}

All usual properties of the Wasserstein metric $W( \theta_1,\theta_2)$ are also true for the distance $d$.
We will use bellow some techniques similar to the ones described in the  gluing lemma of \cite{Vi}, section 7.1.

\begin{theorem}
g is continuous with respect to the  metric $d$, that is,
\[
g(\theta_k)\rightarrow g(\theta),\,\ \text{when}\
d(\theta_k,\theta)\rightarrow 0.
\]
\end{theorem}
\begin{proof}
Let $\pi_k$ and $\pi$ be optimal measures for $g(\theta_k)$ and
$g(\theta)$ respectively. For each $k$, disintegrate the optimal measures
used in the computation of
 $W(\theta_k,\theta)$ and $W(\theta,\theta_k)$ which will be respectively denoted as
\[
\mu_{0,k}=\nu_{0,k} \theta\ \text {and}\
\mu_{k,0}=\nu_{k,0}\theta_k.
\]
Consider the probability $\Xi$  on $\Tt^n\times (\Rr^n)^3$ given by
$$d\, \Xi(x, v_1,v_2,v_3)=d\, \nu_{k,0}(x+ v_1 + v_2,v_3)\pi_k(x+v_1,v_2)\nu_{0,k}(x,v_1)\theta (x).$$
Define the probability $\pi^*\theta$ on $\Tt^n\times \Rr^n$ as
\[
\int_{            \Tt^n\times \Rr^n  }\phi(x,v)\ d(\pi^*\theta) (x,v)=
\int_{\Tt\times (\Rr^n)^3} \phi(x,v_1+v_2+v_3)\ d\, \Xi(x,v_1,v_2,v_3).
\]
Then $\theta$ is the marginal of $ \pi^*\theta $ in the second coordinate. Indeed,
\begin{align*}
&\int_{\Omega}\phi(x+v)\ d(\pi^*\theta)(x,v)=\\
&\int_{\Tt^n\times (\Rr^n)^3}\phi(x+v_1+v_2+v_3) d\, \nu_{k,0}(x+ v_1 + v_2,v_3)\pi_k(x+v_1,v_2)\nu_{0,k}(x,v_1)\theta (x)=\\
&\int_{\Tt^n\times (\Rr^n)^2} \phi(y+v_2+v_3)       d\, \nu_{k,0}(y + v_2,v_3)\pi_k(y,v_2)            \theta_k(y)=\\
&\int_{\Tt\times (\Rr^n)} \phi(z+v_3)      \nu_{k,0}(z,v_3)\theta_k(z)=\int_{\Tt^n} \phi(w)\ \theta(w) =\int_{\Tt^n} \phi(x)
\theta(x).
\end{align*}
Note that $\theta_k$ is the marginal of $\mu_{0,k}$ in the first coordinate and $\theta$ is the marginal of $\mu_{0,k}$ in the second coordinate. Moreover, $\theta$ is the marginal of $\mu_{k,0}$ in the first coordinate and $\theta_k$ is the marginal of $\mu_{k,0}$ in the second coordinate.

Now we need the following lemma.

\begin{lemma}
\[
\int v\ d\mu_{0,k}= -\int v\ d\mu_{k,0},
\]
\end{lemma}
\begin{proof}
Let $\mu\in \Pi[\theta_1,\theta_2]$ be an optimal measure, which we assume first that it is absolutely continuous,  transporting $\theta_1$ to $\theta_2$
with density, that is, $\mu=F(x,y)dx dv$.
Assume that $\theta_1=f(x)dx$ and $\theta_2=g(y)dy$. Then the measure
$\bar \mu:F(x+v,-v)dx dv$ belongs to $\Pi[\theta_2,\theta_1]$
and
\[
\int {|v|^2\over 2}\ d\mu=\int {|v|^2\over 2}\ d\bar \mu.
\]
Since $W(\theta_1, \theta_2)=W(\theta_2,\theta_1)$, $\bar \mu$ is a
minimal measure transporting $\theta_2$ to $\theta_1$.
Moreover, by simple computation
\[
\int v\ d\mu=-\int v\ d\bar \mu.
\]
If $\mu$ is not absolutely continuous we can consider the transformation
$G: \Tt^n \times \Rr^n \to  \Tt^n \times \Rr^n  $ given by $G(x,v)= (x+v,-v)$ (of course, $x+v$ is considered mod 1). Now, we can consider
$\bar \mu = G^* (\mu)$ and use a similar reasoning as before.
\end{proof}

Note that proceeding in the same way as before, we get
$$
\int_{\Tt^n\times \Rr^n}v\ d(\pi^*\theta)=$$
$$
\int_{\Tt^n}\,(v_1+v_2+v_3)\, d\, \nu_{k,0}(x+ v_1 + v_2,v_3)\pi_k(x+v_1,v_2)\nu_{0,k}(x,v_1)\theta (x)=
 Q.$$
Therefore, $\pi^*\in \Pi[\theta, Q]$ and we get finally that
$$\int_{\Tt^n\times \Rr^n} L(x,v)\ d(\pi \theta)\leq \int_{\Tt^n\times \Rr^n} L(x,v)\ d(\pi^* \theta).$$

Now, using the inequality $a\,b\leq \epsilon a^2 + \frac{b^2}{\epsilon}, \,\forall a,b,\epsilon >0,$ and  by Taylor's formula we get for small fixed $\epsilon$
$$
\int_{\Tt^n\times \Rr^n} L(x,v)\ d(\pi^* \theta) =
\int_{\Tt^n\times (\Rr^n)^3} L(x,v_1+v_2+v_3)\ d\,\Xi$$
$$
=\int \,[L(x+v_1,v_2)+L(x,v_1+v_2+v_3)-L(x+v_1,v_2)\,]\ d\Xi=$$
$$
=\int \,L(x+v_1,v_2)+[L(x,v_1+v_2+v_3)- L(x+v_1,v_1+v_2+v_3)] -$$
$$[L(x+v_1,v_2)\, - L(x+v_1,v_1+v_2+v_3)]\ d\,\Xi=$$
$$\int_{\Tt^n\times \Rr^n}L(x,v_2)\, \pi_k(x,dv_2)\,\theta_k(dx)$$
$$+\int \sqrt{\epsilon} L_x(x,v_1 + v_2 + v_3)\cdot (\frac{v_1}{\sqrt{\epsilon}  })\, +O(|v_1|^2)\ d\,\Xi$$
$$-\int L_v(x+v_1,v_2)\cdot (v_1 + v_3)+O(|v_1 + v_3|^2)\ d\,\Xi \leq $$
$$g(\theta_k)+C \,\epsilon + C_\epsilon\,\int_{\Tt^n\times \Rr^n}
[\,{|v_1|^2\over 2}+ {|v_3|^2\over 2}\,]\ d\, \Xi=$$
$$
g(\theta_k)+C\, \epsilon+C_\epsilon \, [W(\theta,\theta_k)+W(\theta_k.\theta)]$$
where $C$ is constant and $C_{\epsilon}$ is a constant which depends on  $\epsilon$.

Given $\delta$ we can choose $\epsilon $, and then $k$, such that  $C \epsilon <\delta/2$ and $ C_\epsilon \, [W(\theta,\theta_k)+W(\theta_k.\theta)]<\delta/2$. Taking $k\to \infty$ we get
\[
g(\theta)\leq \lim_{k\rightarrow \infty}g(\theta_k)+C\epsilon.
\]
As $\epsilon$ goes to zero,
\[
g(\theta)\leq\lim_{k\rightarrow \infty}
g(\theta_k).
\]
We can prove the other inequality in the same way, so $g$ is continuous
with respect to the Wasserstein metric. From this follows the existence of the minimax measure.
\end{proof}

\begin{proposition}
The function $g$ is convex in $\theta$. Furthermore, the function $g_Q(\theta)$ is convex in $Q$ and $\theta$.
\end{proposition}
\begin{proof}
Let $\theta_i$, and $Q_i$, $i=0, 1$, be, respectively, probability measures on $\Tt^n$ and rotation vectors on $\Rr^n$.
Let $\pi_i\in \Pi(\theta_i, Q_i)$. For $0\leq \lambda\leq 1$ define
$\theta_\lambda=(1-\lambda) \theta_0+\lambda \theta_1$, and $Q_\lambda=(1-\lambda) Q_0+\lambda Q_1$.
Let
$\pi_\lambda$ to be the plan in $\Pi(\theta_\lambda, Q_\lambda)$ such that
\[
\theta_\lambda \pi_\lambda=(1-\lambda) \theta_0\pi_0+\lambda \theta_1\pi_1.
\]
Then
\[
g(\theta_\lambda)\leq \int L \theta_\lambda \pi_\lambda=(1-\lambda) \int L \theta_0\pi_0+\lambda \int L \theta_1\pi_1.
\]
By taking the infimum over all plans $\pi_i$ we obtain
\[
g(\theta_\lambda)\leq (1-\lambda) g(\theta_0)+\lambda g(\theta_1).
\]
\end{proof}

\begin{proposition}
$\hat L$ is convex in $Q$.
\end{proposition}
\begin{proof}
It suffices to observe that $\hat L(Q)$ is the supremum of a family of convex functions of $Q$, namely $g_Q(\theta)$.
\end{proof}

\begin{proposition}
There is a maximizer $\theta^*$ of $g(\theta)$ which is point mass, i.e.  $\theta^*=\delta_{x^*}$, for
some $x^*$ in $\Tt^n$.
\end{proposition}
\begin{proof}
Let $\theta^0$ be a maximizer of $g$. The support of $\theta^0$ is contained in $\Tt^n$ which we identify with a cube of side $1$.
We will construct inductively a sequence of maximizing probability measures $\theta^k$ supported in a cube of side $2^{-k}$. Therefore,
these measures will converge in the Wasserstein distance to a measure $\theta^*$ which is supported in a single point and is maximizing
by continuity.

Suppose $\theta^k$ is given and is supported in a cube of side $2^{-k}$. Divide this cube into $2^n$ identical disjoint
cubes with half the sidelenght. Let $\{Q_j\}$ denote the collection of cubes.
Either the restriction to $Q_j$ of $\theta^k$ is zero, in which case we set $\lambda_j=0$,
or
\[
\int_{Q_j} \theta^k=\lambda_j>0.
\]
in which case we set $\theta^k_j=\frac 1 \lambda_j \theta^k 1_{Q_j}$. Note that $\sum \lambda_j=1$, and that each $\theta^k_j$ is a probability measure.
We have $\theta^k=\sum_j \lambda_j \theta^k_j$.
By convexity
\[
g(\theta^k)\leq  \sum_j \lambda_j g(\theta^k_j),
\]
which implies that for every index $j$ for which $\lambda_j>0$, we have that $\theta^k_j$ is a maximizing probability measure. Set $\theta^{k+1}=\theta^k_j$ for
one of those indices. Proceeding inductively we get convergence to a certain $x^*$.

The final conclusion is that there is always point masses which are maximizers.

\end{proof}

If $g_Q$ was strictly convex, then such probability would be unique. In the proof of  Proposition 7.2 and 7.3 we will address this question.

\section{Duality - discrete time}
\label{dualsec}

In this section we proceed in a similar way to \cite{Gom2}.
Fix a probability measure $\theta>0$ and a rotation vector $Q\in \Rr^n$. We will establish that
\begin{align*}
&\inf_{\pi\in \Pi(\theta,Q)}\int  L(x,v)\,d(\pi\,\theta)\,=\\
&\sup_{\phi,P}\int_{\Tt^n} \inf_v
[ L(x,v)+\phi(x+v)-\phi(x)+P\cdot (v-Q)]\,d\theta.
\end{align*}
Set
\[
C^0_*=\left\{\phi\in C(\Omega) : \lim_{|v|\to\infty}
\frac {\phi(x,v)}{|v|}=0\right\},
\]
and observe that the dual of $C^0_*$ is the space $\Mm$
of Radon measures $\mu$ in $\Omega$ with
\[
\int_\Omega |v| d|\mu|<\infty.
\]
Define
\[
h_1(\phi)=\int_{\Tt^n}\sup_{v} \left[-\phi(x,v)-L(x,v)\right]
d\theta.
\]
Let
$$
\Cg=\{  \phi\in C(\Omega) : \phi(x,v)=\psi(x+v)-\psi(x)+P\cdot
(v-Q) \, \text{for some\,} \, \psi\in C(\Tt^n)$$
$$ \text{and for some\,}
P\in \Rr^n \}.
$$
Define
$
h_2(\phi)\, = \,0 $, if $ \phi\in \Cg$, and set $h_2(\phi) \,=\, -\infty $, otherwise.
Let
$$
\Mm_0=\{\mu\in \Mm: \int_\Omega \psi(x+v)-\psi(x) d\mu=0, \forall
\psi\in C(\Tt^n)\ \text{and}$$
$$
\int_{\Omega} P\cdot (v-Q) d\mu=0,\ \forall P\in \Rr^n \}.
$$
Note that the second constraint in the definition of $\Mm_0$ is simply
the rotation vector constraint
\[
\int v d\mu=Q.
\]
Define also
\[
\Mm_1=\left\{\mu\in \Mm: \mu\,=\,\pi\,\theta,\,\pi \ \text{is a non-negative plan, and}\ \int_\Omega d\mu=1\right\}.
\]
As explained before, the constraint that $\pi$ is a non-negative plan simply means that $\pi\theta$ is a (non-negative) probability
measure such that
\[
\int \varphi(x) d (\pi\theta)=\int \varphi(x) d\theta.
\]
\begin{proposition}
\[
h_1^*(\mu)=
\begin{cases}
\int  L d\mu & \text{if}\ \mu\in\Mm_1\\
+\infty & \text{otherwise},
\end{cases}
\]
and
\[
h_2^*(\mu)=
\begin{cases}
0 &\text{if}\ \mu\in\Mm_0\\
-\infty &\text{otherwise}.
\end{cases}
\]
\end{proposition}
\begin{proof}
Recall that
\[
h_1^*(\mu)=\sup_{\phi\in C^0_*}\left[
-\int_\Omega \phi d\mu-h_1(\phi)\right].
\]
If $\mu$ is non-positive then we can choose a sequence of non-negative
functions
$\phi_n\in C^0_*$ such that
\[
-\int_\Omega \phi_n d\mu\to +\infty.
\]
Since $L\geq 0$ we have
\[
h_1(\phi_n)=\int _{\Tt^n}\sup\left[-\phi_n- L\right]d\theta \leq 0.
\]
Therefore $h_1^*(\mu)=+\infty$.
\begin{lemma}
If $\mu\geq 0$ then
\[
h_1^*(\mu)\geq \int_\Omega Ld\mu+\sup_{\psi\in C^0_*}
\left[\int_\Omega \psi d\mu-\int_{\Tt^n}\sup_{v}\psi \,d\theta\right].
\]
\end{lemma}
\begin{proof}  Let $L_n$ be a sequence in $C^0_*$ that increases pointwise
to $L$, $0\leq L_n\uparrow L$.
Any function $\phi\in C^0_*$ can be written as $\phi=-L_n-\psi$ for
some $\psi\in C^0_*$. Therefore
\[
\sup_{\phi\in C^0_*}\left[-\int_\Omega \phi d\mu-h_1(\phi)\right]
=
\sup_{\psi\in C^0_*}\left[\int_\Omega (L_n+\psi) d\mu
-h_1(-L_n-\psi) \right].
\]
Since $L_n-L\leq 0$ we have $\sup(-L+L_n+\psi)\leq \sup \psi$, and so
\[
h_1(-L_n-\psi)\leq \int_{\Tt^n}\sup_v\psi\,d\theta.
\]
Thus
\[
\sup_{\phi\in C^0_*}\left[-\int_\Omega \phi d\mu-h_1(\phi)\right]
\geq \int_\Omega L_nd\mu +
\sup_{\psi\in C^0_*}\left[\int_\Omega \psi d\mu -\int_{\Tt^n}\sup_v\psi \,d\theta\right].
\]
Letting $n\to\infty$, and using the monotone convergence theorem we prove the lemma.
\end{proof}

Now suppose $\int_\Omega d\mu\neq 1$. Then by choosing $\psi=\alpha\in \Rr$
we get
\[
\sup_{\psi\in C^0_*}\left[\int_\Omega \psi d\mu -\int_{\Tt^n}\sup_v\psi \,d\theta \right]
\geq \sup_{\alpha\in \Rr}
\alpha\left(\int_\Omega d\mu-1\right)=+\infty.
\]

If otherwise $\int_{\Omega}d\mu=1$ we have
\[
\int_\Omega (-\phi-L)d\mu\leq \int_{\Tt^n}\sup_v (-\phi-L)\,d\theta=h_1(\phi),
\]
if $\mu=\pi\theta$.
Therefore, for any $\phi$
\[
-\int_\Omega \phi d\mu-h_1(\phi)\leq \int_\Omega Ld\mu,
\]
and so
\[
h_1^*(\mu)\leq \int_\Omega Ld\mu.
\]
If $\int_{\Omega}d\mu=1$ but $\mu\neq\pi\theta$, there exists $\phi_0\in C(\Tt^n)$
such that
\[
\int_{\Omega}\phi_0(x)d\mu\neq\int_{\Tt^n}\phi_0(x)d\theta.
\]
Therefore,
\[
\sup_{\psi\in C^0_*}\left[\int_\Omega \psi d\mu -\int_{\Tt^n}\sup_v\psi \,d\theta \right]
\geq \sup_{\alpha\in \Rr}
\alpha\left(\int_\Omega\phi_0d\mu-\int_{\Tt^n}\phi_0 d\theta\right)
=+\infty.
\]

To compute $h_2^*$ observe that if $\mu\not\in \Mm_0$
then there exists $\psi\in C(\Tt^n)$ such that
\[
\int_\Omega \psi(x+v)-\psi(x) d\mu\neq 0,
\]
or there exists $P\in \Rr^n$ such that
\[
\int_{\Omega} P\cdot (v-Q) d\mu\neq 0.
\]
In any case we can choose a $\hat \phi\in \Cg$ such that
\[
\int_{\Omega}\hat \phi(x,v) d\mu\neq 0.
\]
Thus
\[
\inf_{\phi\in\Cg}-\int_\Omega \phi d\mu\leq \inf_{\alpha\in \Rr}
\int \alpha \hat\phi d\mu=-\infty.
\]
So for $\mu\not\in\Mm_0$ $h_2^*(\mu)=-\infty$. If $\mu\in \Mm_0$
then
\[
\int_\Omega \phi d\mu=0\qquad \forall \phi\in \Cg.
\]
Consequently $h_2^*(\mu)=0$.
\end{proof}

Therefore, in a similar way to \cite{Gom2} we get:

\begin{theorem}
\begin{equation}
\label{maxlq}
\hat L(Q)=\sup_{\theta}\sup_{P,u}
\int \inf_v[L(x,v)+u(x+v)-u(x)+P\cdot (v-Q)]\theta(x).
\end{equation}
\end{theorem}

If we define $\Hh(w)=\sup_v \left(w(x)-w(x+v)-L(x,v)\right)$, we obtain
\[
\hat L(Q)=\sup_{P, u} \sup_x \,[\, -\,\Hh(u)\,-\,P\cdot Q\,].
\]

\section{The dual function and semi-convexity}

In this section, under the assumption that $L$ is quadratic in the velocity,  we consider some properties of the functions $u$ which attains supremum  in the
claim of last Theorem. These results imply, in particular, for each $\theta$ and $Q$, the existence of a maximizing function $u$ for
\eqref{maxlq}. We will ignore the rotation vector constraint $Q$ in order to simplify the notation.



Consider the mapping, the (generalized) double convexification of $u$ (see \cite{Vi} for the related double convexification in optimal transport):
\begin{equation}
\label{dc}
u\mapsto u^{dc}\equiv \sup_w\inf_v \left[
u(x+v-w)+L(x-w,v)-L(x-w,w)
\right].
\end{equation}
\begin{lemma}
Let $u$ be any function, then
\begin{equation}
\label{eqq1}
u^{dc}(x)\leq u(x),
\end{equation}
and
\begin{equation}
\label{eqq2}
\inf_z L(x, z)+u(x+z)=\inf_z L(x, z)+u^{dc}(x+z).
\end{equation}
\end{lemma}
\begin{proof} The inequality \eqref{eqq1}
  follows from choosing $v=w$ in (\ref{dc}). To prove the identity \eqref{eqq2}
observe that for fixed $x$
\begin{multline*}
\inf_z L(x, z)+u^{dc}(x+z)\\
=\inf_z\sup_w\inf_v \left[ L(x, z)+
u(x+z+v-w)\right.\\\left.+L(x+z-w,v)-L(x+z-w,w)
\right]\\
\leq \inf_z L(x, z)+u(x+z),
\end{multline*}
by choosing for each $w$, $v=w$. In a similar way, by choosing, for each $z$,
$w=z$ we obtain the opposite inequality.
\end{proof}


The result above shows that we can look  for maximizers in a smaller class as we can assume that
any maximizer is the double convexification of a function $u$. We apply this result to show in next proposition that
we can therefore take the maximizers
$u$ with a bounded semi-convexity constant, as long as $L$ satisfies suitable hypothesis.
Therefore, it will follow that there exists a maximizer $u$ with a bounded convexity modulus. In fact, for a fixed $\theta$ we consider a sequence $u_n$ such that almost realize
$$\sup_{u}
\int \inf_v[L(x,v)+u(x+v)-u(x)]\theta(x).$$
Then, we can extract a convergent subsequence it is clear that the limit $u$ of this sequence is a maximizer with bounded convexity modulus.

\begin{proposition} Suppose $L(x,v)$ satisfies, for any $x$, $\bar w$ and $y$,
\[
-L(x ,  \bar{w} +y     )-
L(x,  \bar{w} -y     )+2L(x,  \bar{w}      )\geq -C|y|^2.
\]
Then, for any periodic function $u:\Tt^n\to \Rr$, the function $u^{dc}$ is semiconvex, that is,
there exists a constant $C$ such that
\[
u^{dc}(x+y)+u^{dc}(x-y)-2 u^{dc}(x)\geq -C|y|^2.
\]
\end{proposition}
\begin{proof}
Given $u$ and $x$ fixed, suppose $\bar{w}$ is such that
$$u^{dc}(x)  = \inf_v \left[
u(x+v-   \bar{w}   )+L(x-  \bar{w}     ,v)-L(x-  \bar{w}       ,  \bar{w}      ) \right]        .$$
Now we will estimate  $u^{dc}(x+y)$ and   $u^{dc}(x-y) $.
Taking $w= \bar{w} +y$
we get
$$  u^{dc}(x+y) \geq \inf_v \left[ \,u(x+v-   \bar{w}   )+\,L(x-  \bar{w}     ,v)\,-\,\, L(x-  \bar{w}       ,  \bar{w} +y     ) \right]    .$$
Taking $w= \bar{w} -y$
we get
\[
u^{dc}(x-y) \geq  \inf_v u(x+v-   \bar{w}   )+\,L(x-  \bar{w}     ,v)\,-\,\, L(x-  \bar{w}       ,  \bar{w} -y     ).
\]

Denote $v_1$ and $v_2$  vectors such that, respectively
$$  u^{dc}(x+y) \geq   \,u(x+v_1-   \bar{w}   )+\,L(x-  \bar{w}     ,v_1)\,-\,\, L(x-  \bar{w}       ,  \bar{w} +y     )    ,$$
and
$$  u^{dc}(x-y) \geq   \,u(x+v_2-   \bar{w}   )+\,L(x-  \bar{w}     ,v_2)\,-\,\, L(x-  \bar{w}       ,  \bar{w} -y     ).$$
Then,
$$ - \,u^{dc}(x)  \,\geq
 \,-\,u(x+v_1-   \bar{w}   )-\,\,L(x-  \bar{w}     ,v_1)\,+\,\, L(x-  \bar{w}       ,  \bar{w}      ),$$
$$ - \,u^{dc}(x)  \,\geq
 \,-\,u(x+v_2-   \bar{w}   )-\,\,L(x-  \bar{w}     ,v_2)\,+\,\, L(x-  \bar{w}       ,  \bar{w}      ).$$
Now, adding the last four expressions we get
$$  u^{dc}(x+y) +  u^{dc}(x-y) - 2 \, \,u^{dc}(x)\,\geq $$
$$\, -\,\, L(x-  \bar{w}       ,  \bar{w} +y     )\,-\,
L(x-  \bar{w}       ,  \bar{w} -y     )\,+\,2\, L(x-  \bar{w}       ,  \bar{w}      )  .$$
From this follows that
$$u^{dc}(x+y)+u^{dc}(x-y)-2 u^{dc}(x)\geq -C|y|^2.
$$

\end{proof}


\section{Minimax stationary Mather measures in continuous time}

In this section we
consider minimax stationary Mather measures in continuous time. Although these could seem the natural generalization of
the previous problems, we will
give a few examples which illustrate the main problems and motivate the definition and study of
minimax periodic Mather measures, in the next section.

\begin{Definition}
We say $ \mu=\theta(x)\pi(x,v)$ in $\Tt^n\times \Rr^n$ is holonomic (continuous time setting) if for any given $C^1$ function
$\varphi:\Tt^n\to \mathbb{R}$, we have
$$ \int v\, D_x\varphi\, d \pi(x,dv) \, \theta(dx) =0.$$
We denote the set of such probabilities by ${\cal H}$.
\end{Definition}
For a given probability $\theta$ over $M$ we denote $\Pi(\theta)$ the set of admissible plans
$\pi(x,v)$ on $TM$ such that $\theta \pi\in {\cal H}$.

Define
$$ g(\theta)\, =\, \inf_{\pi\in \Pi(\theta)}\, \int_{\Tt^n\times \Rr^n} L(x,v)\,
\pi(x,dv)\, \theta\,(dx).$$
For $\theta$ fixed, we denote by $\pi_\theta$ any solution of the minimization problem above.
Any probability measure $\theta_L$ which attains the supremum
of $g(\theta)$
is called a (continuous time) stationary minimax Mather measure, and sometimes, to simplify
notation,
we will also call $\theta_L \pi_L$, with $\pi_L\in \Pi(\theta_L)$, a minimax Mather measure.


We remark here that, as before, the functions $g(\theta)$ and $g_Q(\theta)$ are convex functions
of $\theta$ or $Q$ and $\theta$, respectively.

\begin{proposition}
If $\mu\,=\, \theta\,(dx) \, \pi(x,dv)$ is a minimax Mather measure then
there exists a function $v(x):\Tt^n\to \Rr^n$ such that
$\mu$ has support in a the graph $(x, v(x))$.
\end{proposition}
\begin{proof}
This proof is similar, for instance,
 to the one in Theorem 3 in \cite{BG}, which considers the classical continuous time Aubry-Mather problem.
%
For each $x$, let $v(x) = \int v d \pi(x,dv)$ and $\eta(x,dv) = \delta_{v(x)} (dv).$
From the strict convexity we get that for each fixed $x$
$$ \int_{\mathbb{R}^d} L(x,v) d \, \eta(x\,, d\, v)\,<\,
\int_{\mathbb{R}^d} L(x,v) d \, \pi(x ,d\, v),$$
for any point $x$ where the probability  $ \delta_{v(x)} (dv)$ is different from $ \pi(x,dv).$
The probability $\theta( dx)\, \eta(x,dv)$ is holonomic. From, this it follows that $\mu$ has support on a graph.
\end{proof}


\begin{proposition}

The only rotation number for which there can exist a  minimax stationary Mather measure is $Q=0$.

\end{proposition}

\begin{proof}
Since the function $g_Q(\theta)$ is convex, applying the same reasoning as before, if there exists
a maximizing measure, there exists a maximizing measure $\bar \theta$ supported in a single point.
From the graph theorem we conclude the corresponding minimax stationary Mather measure has the form $\delta_{x_0}(x) \delta_{v_0}(v)$.
It is clear also that unless $v_0=0$ this measure is not holonomic. Thus its rotation number must be $0$. This shows that
the only rotation number for which there can exist a  minimax stationary Mather measure is $Q=0$.
\end{proof}

\section{Minimax periodic Mather measures}

To overcome the non-existence issues in the previous section and study the continuous time problems,
 we consider the following setting: let $\theta$ be a fixed probability on $\Tt^n $, and we define
\[
g(\theta)=\inf_{\rho} \int_0^T \int_{\Tt^n\times \Rr^n} L(x, v )d \rho(x, v,t),
\]
over all measures $\rho$ on $[0,T]\times \Tt^n\times \Rr^n$ which satisfy, for all smooth $\varphi(x,t)$, $x \in \Tt^n$, $t\in [0,T]$,
\begin{equation}
\label{hnp}
\int_0^T \int_{\Tt^n\times \Rr^n} \varphi_t+v D_x\varphi \,d\rho=\int_{\Tt^n} \varphi(x,T)-\varphi(x,0)\,d \theta.
\end{equation}

We denote such set by $\Pi(T,\theta)$.
We may as well add the rotation number constraint
\begin{equation}
\label{rnp}
\int_0^T \int_{\Tt^n\times \Rr^n} v d\rho =Q.
\end{equation}
We denote by $\Pi(T,\theta,Q)$ the set of measures
$\rho$ on $[0,T]\times \Tt^n\times \Rr^n$
which satisfy the two constraints above. Using the same notation as before, we consider $g_Q^T(\theta)$ for the minimization of $\rho \in \Pi(T, \theta,Q)$.

We refer the reader to \cite{BB} for several results on Mather theory which are similar to the minimax setting we consider here (for autonomous Lagrangians). In the notation of \cite{BB} we are considering the set of initial transport measures on ${\cal I}(\theta,\theta)$ and $\rho$ is a transport measure (see definition 5 on that paper).

\begin{proposition}
For every probability measure $\theta$ on $ \Tt^n$ and every rotation number $Q$ there exists a measure $\rho_Q$ which satisfies \eqref{hnp} and \eqref{rnp}.
\end{proposition}
\begin{proof}
If $QT\in \Zz^n$ then it suffices to consider
$\rho_Q(t,x,v)= \frac{1}{T} \, \delta(v-Q)\, \theta(x).$
Otherwise
we can always write $Q$ as a convex combination $Q=\sum_i\lambda_iQ_i$ of vectors $Q_i\in \Rr^n$. Then define
\[
\rho_Q=\sum \lambda_i \rho_{Q_i}.
\]
Therefore, $\Pi(T,\theta,Q)$ is not empty.
\end{proof}

The minimax periodic Mather problem consists in maximizing
\[
\hat L(Q)=\sup_{\theta} g_Q^T(\theta).
\]

We call the  measure $\rho$ which realizes such problem  of $T$-minimax probability.

By convexity on $v$ and using a standard weak convergence argument, we can prove that for each $\theta$ there exists $\rho$ such that
$$g_Q(\theta)=\int_0^T \int_{\Tt^n\times \Rr^n} L(x, v )d \rho(x, v,t).$$

Consider a sequence $\theta_n$  such that $g_Q(\theta_n) \to \hat L$, when $n \to \infty$.  One can consider weak limits of subsequences of the probabilities $\theta_n$ over $\Tt^n$, and getting in this way a limit probability measure which we denote by $\theta$.
In the same way as before (discrete time case) we want to show that $g_Q(\theta_n) \to g_Q(\theta)$, whenever $\theta_n\rightharpoonup \theta$.
Assume for now that this is true.
Then Given a certain $\theta$ there exist a minimizer $\rho$ for $g_Q(\theta)$.  Then, $\rho$ is a minimax probability for such $Q$.

In order to show that  $g_Q(\theta_n) \to g_Q(\theta)$,
we will consider once more the metric $d(\theta_1,\theta_2)$, for     $ \theta_1,\theta_2\in       P(\Tt^n)$,  defined before.

\begin{proposition}
$g_Q^T$ is continuous on $\theta$.
\end{proposition}
\begin{proof}
The proof is similar to the one in section 4.
We describe the main idea, omitting the details.  Given $\epsilon$, and $\theta_0$ and $\theta_1$ whose Wasserstein distance is suitably small there is
a transport measure $\rho_{01}$ and $\rho_{10}$  in time $\epsilon$ such that
\[
\int L d\rho_{01}, \int L d\rho_{10}<\epsilon.
\]
Also given a measure $\rho_1$ which is a minimizer for $g(\theta_1)$ in time $T$ we can build another measure $\rho_1^\epsilon$ on $[\epsilon, T-\epsilon]$
such that
\[
\int L d\rho_1^\epsilon< g(\theta_1)+C \epsilon
\]
Then we consider the concatenation of $\rho_{01}, \rho_1, $ and $\rho_{10}$ and we obtain
\[
g(\theta_0)\leq g(\theta_1)+C \epsilon.
\]
\end{proof}

From this result and the fact that $\Pi(T,\theta,Q)$ is not empty, we get finally the existence of a minimax measure
$\rho$ for $g_Q$ in $\Pi(T,\theta,Q)$.

\begin{proposition}
$g(\theta)$ and
$g_Q(\theta)$ convex, resp. on $\theta$ and $\theta$ and $Q$.
\end{proposition}
\begin{proof}
We will consider the case of $g_Q$ as the proof for $g$ is similar.
Let $\theta_i$, probability measures on $\Tt^n$ with rotation vector $Q$  on $\Rr^n$.
Let $\rho_i\in \Pi(T,\theta_i, Q)$. For $0\leq \lambda\leq 1$ define
$\theta_\lambda=(1-\lambda) \theta_0+\lambda \theta_1$.
Let
$\rho_\lambda$ be the plan in $\Pi(T,\theta_\lambda, Q)$ such that
\[
\rho_\lambda=(1-\lambda) \rho_0+\lambda \rho_1.
\]
Then
\[
g_Q(\theta_\lambda)\leq \int_0^T \int L \rho_\lambda =(1-\lambda) \int _0^T \int L \rho_0+\lambda \int_0^T\int L \rho_1.
\]
By taking the infimum over all plans $\rho_i$ we obtain
\[
g_Q(\theta_\lambda)\leq (1-\lambda) g_Q(\theta_0)+\lambda g_Q(\theta_1).
\]
\end{proof}

\begin{theorem}
\label{84}
For a fixed $T>0$,
there exists a
minimax Mather measure $\theta$ for $g_Q$ which is supported in a single point.
\end{theorem}
\begin{proof}
The proof uses again a convexity argument and is analogous to the  one of the last proposition of section 4.
\end{proof}

A similar result is true for $g$ over $\Pi(T,\theta)$.

\section{Additional examples}

Consider in $T \Tt^2$ the Lagrangian
\[
L(x,v) =
L(x_1,x_2,v_1,v_2)=\frac{v_1^2+v_2^2}{2}  + v_1.
\]
First we consider the case without rotation number constraint.

As we have discussed before, the minimax measure problem can be analyzed by considering minimax orbits
associated to the time $T$. In other words, from theorem \ref{84}, we just have to
consider probability measures $\theta$ supported in a single point, that is,
of the form $\theta= \delta_x$, for each $x\in\mathbb{T}^2$.
The plan $\rho(x,v,t)$ is obtained by linear superposition of plans associated with
curves $\gamma:[0,T]\to \Tt^2$ which are solution to the
Euler-Lagrange equation (see \cite{BB}) and satisfy periodic conditions $\gamma(0)=\gamma(T)=x$.
Since this Lagrangian only depends on the velocity,
 we known that $\gamma$ is either a straight line with
constant velocity or a constant trajectory $\gamma(t)=x$, for all
$t\in[0,T]$. In the last case the action is zero. Denote by
$n_1(\gamma)=\int_\gamma v_1= \int_\gamma dx_1$ the algebraic number
of turn around  in the $x_1$ direction. In this case $\gamma$ is
non constant, because $\gamma$ has to be periodic. Because of minimality,
it is clear that $v_2=0$.
Thus it is a horizontal
line and the velocity has constant value equal to $
l(\gamma)/T$, where $l(\gamma)$ is the length of the curve. In this case we have several different  measures on the tangent bundle associated to different horizontal lines. The action of
$\gamma$ is $\int_\gamma L(x,v) = T \, \frac{l(\gamma)^2}{2\,
T^2} + n(\gamma)= \frac{l(\gamma)^2}{2\, T} + n(\gamma)$. If the
trajectory $\gamma$ turns around in the $x_1$ direction $n(\gamma)$
times over a straight line trough $x$, then the action is
$$\int_\gamma L(x,v) =\frac{n(\gamma)^2}{2\, T} + n(\gamma).$$

For $1/2< T<3/2$, the value $n(\gamma)=-1$ is optimal. For $0<T<1/2$
the optimal value is  $n(\gamma)=0$ and we get the constant
trajectory. For $T> 3/2$ one can get values $n(\gamma)<-1$, as $n(\gamma)\sim -T$ for large $T$.

Note that the properties described above are independent of $x$.
Therefore, for example, for $1/2<T<3/2$,  the maximization of
$g(\theta)$ gives all different possibilities of horizontal lines
trough $x$, with $x\in \mathbb{T}^2$. Then, in this case, for each
fixed $T$ we get minimal plans which are convex combination of a
continuum of probabilities. In this case we do not have uniqueness.

We point out the difference of the minimax problem to the usual
Mather problem (in which the period $T$ is not fixed) in the present
case. The Mane critical value is $c(L)=1/2$ and the minimizing
probabilities are given by $1$-periodic curves $\gamma$ which are horizontal straight lines which satisfy
$n(\gamma)=-1$.

Now consider the case of a fixed a vector $Q=(Q_1,Q_2)\in \Rr^2$ and we look for $Q$-minimax measures. As before, we assume
$\theta=\delta_x$.
As we haved pointed out before, the plan transport plan $\rho(x, v, t)$ can consist on a superposition of transport plans associated with
several trajectories solutions to
the Euler-Lagrange equations. Then one has the family of $T$-periodic curves $\gamma_k$ solving the Euler-Lagrange equation
passing through $x$. These curves
have constant velocity, and are indexed by
$k\in \Zz^2$, where
\[
k=n(\gamma_k)=\int_0^T v dt.
\]
is the algebraic number of turns. Clearly there exists $0\leq \lambda_k\leq 1$ with $\sum_k \lambda_k=1$ so that the minimax transport
plan can be written as
\begin{equation}
\label{finitesum}
\sum_k \lambda_k \rho_k(x, v, t),
\end{equation}
where $\rho_k$ is the transport plan associated with the curve $\gamma_k$. Since the action of $\gamma_k$ is
\[
\frac{|k|^2}{2 T}+k_1,
\]
the sum in \eqref{finitesum} is a finite sum. Also note that the mapping $k\mapsto \frac{|k|^2}{2 T}+k_1$ is strictly convex,
therefore if for some $k^*=Q$ then $\lambda_{k^*}=1$. For all other values
\[
\hat L(Q)=\inf_{\lambda} \sum_k \lambda_k \rho_k(x, v, t),
\]
under the constraint $\sum_k \lambda_k k=Q$.

\medskip

As a second example, consider
a general Lagrangian $L(x, v)$ on $T\Tt^n$.  Fix $T>0$.
For each $k\in \Zz^n$
and any $x\in \Tt^n$
look for a minimal orbit starting at $x$ and ending at $x$ with rotation number $k$, $\gamma_k$, and let $S_k(x)$
be the action of such an orbit. Note that this orbit may not be a periodic solution to the Euler-Lagrange equation.

Define
\[
g_Q(x)=\inf_{\lambda} \sum \lambda_k S_k(x),
\]
where $\lambda$ is constrained to $0\leq \lambda_k\leq 1$, $\sum_k \lambda_k=1$, and
\[
\sum_k \lambda_k k =Q.
\]
Then
\[
\hat L(Q)=\sup_x g_Q(x).
\]


\begin{thebibliography}{99}

\bibitem[AM] {AM} 
F. Alessio, Francesca and P. Montecchiari.
{\it Entire solutions in $ \mathbb{ R}^2$ for a class of Allen-Cahn equations}. 
ESAIM Control Optim. Calc. Var. 11, no. 4, 633-672, 2005

\bibitem[Ambro] {Ambro}  L. Ambrosio,  N. Gigli,  G. Savare {\slshape ``Gradient Flows in Metric Spaces and in the Space of
Probability Measures.''}, ETH  Birkhauser Verlag, Basel, 2005.

\bibitem[Be] {Be}
U. Bessi, {\it Many solutions of elliptic problems on $\mathbb{ R}^n $ of irrational slope}. Comm. Partial Differential Equations 30, no. 10-12, 1773-1804, 2005


 \bibitem [BG] {BG} A. Biryuk and D. A. Gomes. {\it An introduction to Aubry-Mather Theory}.  Preprint 2007, to appear in {\it S\~ao Paulo Journal of Mathematical Sciences}.

 \bibitem [BB] {BB} P. Bernard and B. Buffoni . {\it Optimal mass transportation and Mather theory}.  J. Eur. Math. Soc. 9, no. 1,  85-121, 2007.


 \bibitem [CS] {CS} P. Cannarsa and C. Sinestrari. {\it
     Semiconcave functions, Hamilton-Jacobi equations, and optimal
     control}. Birkhäuser Boston Inc., Boston, MA, 2004.

 \bibitem [CI] {CI} G. Contreras and R. Iturriaga.  {\it Global Minimizers of
 Autonomous    Lagrangians}. AMS  2004. To appear.

\bibitem [CP] {CP} G. Contreras and G. Paternain. {\it Connecting orbits between static classes for generic Lagrangian systems}. Topology 41, pp 645-666, 2002.

 \bibitem [Fa] {Fa} A. Fathi.  {\it Weak KAM Theorem and
     Lagrangian Dynamics}. Cambridge University Press 2004. To appear.

\bibitem [LV] {LV}
R. de la Llave and E. Valdinoci. {\it Ground states and critical points for generalized Frenkel-Kontorova models in $\Bbb Z\sp d$}. Nonlinearity 20 , no. 10, 2409--2424, 2007

\bibitem [LV1] {LV1}
R. de la Llave and E. Valdinoci.
{\it Multiplicity results for interfaces of Ginzburg-Landau-Allen-Cahn equations in periodic media}. Adv. Math. 215, no. 1, 379-426, 2007

\bibitem [DM] {DM}
C. Dellacherie and P.-A. Meyer, Probabilities and potential, North-Holland Publishing Co., Amsterdam,
1978.

 \bibitem [L1] {L1} L. C. Evans. {\it Towards a Quantum Analog of Weak KAM Theory}. Comm. in Math. Phys. 244, 311-334, 2004.

 \bibitem [Gom1] {Gom1} D. A. Gomes and C. Valls. {\it Wigner Measures and Quantum Aubry-Mather Theory}. Asymptotic Analysis, 51, no. 1, 47-61, 2007.

 \bibitem [Gom2] {Gom2} D. A. Gomes. {\it Viscosity solution methods
     and discrete Aubry-Mather problem}. Discrete Contin. Dyn. Syst.,
   13(1): 103-116, 2005.

\bibitem [GLM] {GLM} D. A. Gomes, A. O. Lopes and J. Mohr.     {\it  The
Mather measure and a Large Deviation Principle for the Entropy Penalized Method}, 2007, to appear.



 \bibitem [GV] {GV} D. A. Gomes, E. Valdinoci. {\it Entropy
     Penalization Methods for Hamilton-Jacobi Equations}.  Adv. Math. 215, No. 1, 94-152, 2007.
     
\bibitem [MMW] {MMW}
F. Mahmoudi, A. Malchiodi and J. Wei, {\it Transition layer for the heterogeneous Allen-Cahn equation}. Ann. Inst. H. Poincaré Anal. Non Linéaire 25, no. 3, 609--631, 2008

\bibitem[Mat1] {Mat1}
J. Mather. {\it  A criterion for the nonexistence of invariant circles.} Inst. Hautes Études Sci. Publ. Math. No. 63, 153-204, 1986

\bibitem[Mat] {Mat} J. Mather. {\it Action minimizing invariant measures for positive definite Lagrangian systems}. Math. Z., N 2, 169-207, 1991.


\bibitem [RS] {RS}
P. H. Rabinowitz and E. Stredulinsky, {\it On some results of Moser and of Bangert}. Ann. Inst. H. Poincaré Anal. Non Linéaire 21, no. 5, 673-688, 2004

\bibitem [Ra] {Ra} S. Rachev and L. Ruschendorf, {\it Mass transportation problems}, Vol I, Springer Verlag, 1998.


 \bibitem [Vi] {Vi} C. Villani, {\it Topics in Optimal Transportation}, AMS, 2003.

\end{thebibliography}
\end{document}